\documentclass[11pt]{amsart}
\usepackage{graphicx,subcaption}
\usepackage{amsmath} 
\usepackage{amsthm,amsfonts,amssymb,mathrsfs,amscd,amstext,amsbsy} 
\usepackage{epic,eepic} 
\usepackage{yfonts}
\usepackage{paralist,enumerate}
\usepackage[all]{xy}
\usepackage{hyperref}
\hypersetup{colorlinks}

\usepackage{tikz}
\usetikzlibrary {positioning}

\newtheorem{theorem}{Theorem}[section]
\newtheorem{cor}[theorem]{Corollary}
\newtheorem{lemma}[theorem]{Lemma}

\newtheorem{lem}[theorem]{Lemma}
\newtheorem{pro}[theorem]{Proposition}

\newtheorem{definition}[theorem]{Definition}
\newtheorem*{Definition*}{Definition}
\numberwithin{equation}{section}

\def\qed{\hfill \ifhmode\unskip\nobreak\fi\quad\ifmmode\Box\else$\Box$\fi\\ }

\begin{document}

\title[8-dimensional almost complex $S^1$-manifold with 4 fixed points]{Circle actions on 8-dimensional almost complex manifolds with 4 fixed points}
\author{Donghoon Jang}
\thanks{Donghoon Jang is supported by Basic Science Research Program through the National Research Foundation of Korea(NRF) funded by the Ministry of Education(2018R1D1A1B07049511).}
\address{Department of Mathematics, Pusan National University, Pusan, Korea}
\email{donghoonjang@pusan.ac.kr}

\begin{abstract}
Consider a circle action on an 8-dimensional compact almost complex manifold with 4 fixed points. To the author's knowledge, $S^2 \times S^6$ is the only known example of such a manifold. In this paper, we prove that if the circle acts on an 8-dimensional compact almost complex manifold $M$ with 4 fixed points, all the Chern numbers and the Hirzebruch $\chi_y$-genus of $M$ agree with those of $S^2 \times S^6$. In particular, $M$ is unitary cobordant to $S^2 \times S^6$.
\end{abstract}

\maketitle

\tableofcontents

\section{Introduction}

$\indent$

The purpose of this paper is to discuss the classification of circle actions on almost complex manifolds with discrete fixed point sets and give a classification if the dimension of the manifold is eight and if there are exactly four fixed points.

Let the circle act on a $2n$-dimensional almost complex manifold $(M,J)$. Throughout this paper, any circle action on an almost complex manifold $(M,J)$ is assumed to preserve the almost complex structure $J$. Let $p$ be an isolated fixed point. Then the local action of the circle near $p$ can be identified with
\begin{center}
$g \cdot (z_1, \cdots, z_n) = (g^{w_{p,1}}z_1, \cdots, g^{w_{p,n}} z_n)$,
\end{center}
where $g \in S^1 \subset \mathbb{C}$, $z_i$ are complex numbers, and $w_{p,i}$ are non-zero integers. The non-zero integers $w_{p,1},\cdots,w_{p,n}$ are called the \textbf{weights} at $p$. If the fixed point set is discrete, the weights at the fixed points determine the Hirzebruch $\chi_y$-genus of $M$ (which determines the Todd genus, the signature, and the Euler characteristic of $M$), and the Chern numbers of $M$ (which determine the equivariant cobordism class of $M$).

We discuss the classification of circle actions on compact almost complex manifolds that have few fixed points. For this, let the circle act on a compact almost complex manifold $M$.

First, suppose that there is exactly one fixed point. Then $M$ must be the point.

Second, suppose that there are exactly two fixed points. Then either $M$ is the 2-sphere, or $\dim M=6$ and the weights at the fixed points agree with a circle action on $S^6$; see Theorem \ref{t211}. This is proved in \cite{J3}, whose proof closely follows the proof for a holomorphic vector field on a complex manifold with simple isolated zeroes by Kosniowski \cite{K1}. Pelayo and Tolman proved an analogous result for symplectic $S^1$-actions on symplectic manifolds \cite{PT}.

Third, suppose that there are exactly three fixed points. Then the manifold must be four dimensional, and the weights at the fixed points agree with a standard linear action on $\mathbb{CP}^2$; see Theorem \ref{t211}. The result is announced in \cite{J3}, which is proved by carefully adapting the proof for a symplectic circle action on a compact symplectic manifold with 3 fixed points in \cite{J1}.

Fourth, suppose that there are exactly four fixed points. If $\dim M=2$, $M$ is a disjoint union of two $S^2$'s, each of which is equipped with a rotation. If $\dim M=4$, the weights at the fixed points agree with a circle action on a Hirzebruch surface; see Theorem \ref{t212}. If $\dim M=6$, six types occur for the weights at the fixed points, $\mathbb{CP}^3$ type, complex quadric in $\mathbb{CP}^4$ type, Fano 3-fold type, $S^6 \cup S^6$ type, blow up of a fixed point of a rotation on $S^6$-type, and one unknown type. These results are proved in \cite{J2}. For related results, see \cite{A} and \cite{T}.

Now, let us consider the case of an 8-dimensional compact almost complex manifold with 4 fixed points. To the author's knowledge, $S^2 \times S^6$ is the only example of such a manifold. In this paper, we prove that if an 8-dimensional manifold $M$ has 4 fixed points, $M$ has the same Chern numbers and Hirzebruch $\chi_y$-genus as $S^2 \times S^6$; in particular, $M$ is unitary cobordant to $S^2 \times S^6$.

\begin{theorem} \label{t11} Let the circle act on an 8-dimensional compact almost complex manifold $M$ with 4 fixed points. Then the Hirzebruch $\chi_y$-genus of $M$ is $\chi_y(M)=-y+2y^2-y^3$ and the Chern numbers of $M$ are 
\begin{center}
$\displaystyle \int_M c_1^4=\int_M c_1^2 c_2=\int_M c_2^2=0$, and $\displaystyle \int_M c_1 c_3=\int_M c_4=4$. 
\end{center}
In particular, $M$ is unitary cobordant to $S^2 \times S^6$. \end{theorem}

We will prove Theorem \ref{t11} in Section \ref{s3}. In \cite{S}, Sabatini raised a question what all the possible values of the Chern numbers are for a circle action on a compact almost complex manifold with a discrete fixed point set. Theorem \ref{t11} answers the question if the dimension of the manifold is 8 and if there are 4 fixed points; in this case, the Chern numbers are unique. On the other hand, if the dimension is eight and if there are 5 fixed points, the Chern numbers may not be unique; see \cite{GS}.

The ideas of the proof of Theorem \ref{t11} are to use the Atiyah-Singer index theorem and the ABBV localization theorem (and their consequences), and to utilize a multigraph associated to $M$ by means of the fixed point set and the weights at the fixed points. The necessary background and preliminaries are given in Section \ref{s2}.

Unlike the results for fewer fixed points or for lower dimensions, we do not classify the weights at the fixed points. Classifying the weights at the fixed points would involve many cases to consider and each case would involve a lengthy computation.

Beyond 4 fixed points, in the case of $\dim M=8$ with 5 fixed points, Godinho and Sabatini determined the (equivariant) Chern classes and the (equivariant) cohomology ring of a Hamiltonian circle action on an 8-dimensional compact symplectic manifold with 5 fixed points under an additional assumption. In \cite{JT}, it was shown that the additional assumption always holds for such a manifold.

A natural question is whether or not there exists a circle action on a compact almost complex manifold with 4 fixed points in higher dimensions. While $S^6 \times S^6$ provides a 12-dimensional example, to the author's knowledge, in dimensions higher than 8 and other than 12, there are no known examples. We discuss the question in more details in Section \ref{s4}. The question is related to Kosniowski's conjecture, which asserts that if the circle acts on a compact unitary manifold $M$ with isolated fixed points and $M$ does not bound a unitary manifold equivariantly, then $M$ has at least $\frac{\dim M}{4}+1$ fixed points \cite{K2}. 

\section{Background and preliminaries} \label{s2}

In this section, we give background and provide necessary materials to prove our main result, Theorem \ref{t11}. 

Let the circle act on a compact oriented manifold $M$. The \textbf{equivariant cohomology} of $M$ is
\begin{center}
$H_{S^1}^*(M) := H^*(M \times_{S^1} S^{\infty})$. 
\end{center}
The projection map $\pi:M \to \{\textrm{pt}\}$ induces a push-forward map
\begin{center}
$\displaystyle \int_M := \pi_* : H_{S^1}^i (M;\mathbb{Z}) \longrightarrow H^{i - \dim M} (\mathbb{CP}^\infty ; \mathbb{Z})$.
\end{center}
The Atiyah-Bott-Berline-Vergne localization theorem states that the push-forward map can be computed by the fixed point data.

\begin{theorem} [ABBV localization theorem]\cite{AB} \label{t21} Let the circle act on a compact oriented manifold $M$. For any $\alpha \in H_{S^1}^*(M;\mathbb{Q})$,
\begin{center}
$\displaystyle \int_M \alpha = \sum_{F \subset M^{S^1}} \int_F \frac{\alpha|_F}{e_{S^1}(N_F)}$.
\end{center}
Here, the sum is taken over all fixed components, and $e_{S^1}(N_F)$ is the equivariant Euler class of the normal bundle of $F$.
\end{theorem}

The Hirzebruch $\chi_y$-genus is the genus associated to the power series $\frac{x(1+ye^{-x(1+y)})}{1-e^{-x(1+y)}}$. For a $2n$-dimensional compact almost complex manifold $M$, let $\chi_y(M) = \sum_{i=0}^n \chi^i(M) \cdot y^i$ be the Hirzebruch $\chi_y$-genus of $M$. Here, $\chi^i(M)=\int_M T_i^n$, where $T_i^n$ is a rational combination of products of Chern classes $c_{j_1} \cdots c_{j_k}$ such that $j_1+\cdots+j_k=n$. Note that $\chi_{-1}(M)$ is the Euler characteristic of $M$, $\chi_0(M)=\textrm{Todd}(M)$ is the Todd genus of $M$, and $\chi_1(M)=\textrm{sign}(M)$ is the signature of $M$.

Let the circle act on a compact almost complex manifold with a discrete fixed point set. As an application of the Atiyah-Singer index theorem, Li obtained the following formula, whose proof follows the proof in \cite{HBJ} for a complex manifold.

\begin{theorem} \label{t22} \cite{HBJ}, \cite{L} Let the circle act on a $2n$-dimensional compact almost complex manifold $M$ with a discrete fixed point set. Then for each $i$ such that $0 \leq i \leq n$,
\begin{center}
$\displaystyle \chi^i(M) = \sum_{p \in M^{S^1}} \frac{\sigma_i(g^{w_{p,1}},\cdots,g^{w_{p,n}})}{\prod_{j=1}^n (1-g^{w_{p,j}})}=(-1)^i N_i=(-1)^i N_{n-i}$.
\end{center}
Here, $\chi_y(M)=\sum_{i=0}^n \chi^i(M) \cdot y^i$ is the Hirzebruch $\chi_y$-genus of $M$, $g$ is an indeterminate, $\sigma_i$ is the $i$-th elementary symmetric polynomial in $n$ variables, and $N_i$ is the number of fixed points that have exactly $i$ negative weights.
\end{theorem}

Consider a circle action on a compact almost complex manifold with a discrete fixed point set. In \cite{H}, Hattori proved that for each integer $w$, the number of times $w$ occurs as weights over all fixed points, counted with multiplicity, is equal to the number of times $-w$ occurs as weights over all fixed points, counted with multiplicity. Li reproved this in \cite{L}; also see \cite{PT} for a symplectic action.

\begin{pro} \label{p23} \cite{H}, \cite{L}
Let the circle act on a compact almost complex manifold $M$ with a discrete fixed point set. For each integer $w$,
\begin{center}
$\displaystyle \sum_{p \in M^{S^1}} N_p(w)=\sum_{p \in M^{S^1}} N_p(-w)$,
\end{center}
where $N_p(w)=|\{i \ | \ w_{p,i}=w, 1 \leq i \leq n\}|$ is the number of times $w$ occurs as weights at $p$.
\end{pro}

Proposition \ref{p23} has the following corollary.

\begin{cor} \cite{H} \label{c24} Let the circle act on a $2n$-dimensional compact almost complex manifold $M$ with a discrete fixed point set. Then 
\begin{center}
$\displaystyle \sum_{p \in M^{S^1}} \sum_{i=1}^n w_{p,i}=0$. 
\end{center} \end{cor}

Proposition \ref{p23} enables us to associate a labeled, directed multigraph to a circle action on a compact almost complex manifold with a discrete fixed point set. First, assign a vertex to each fixed point. For a positive integer $w>0$, Proposition \ref{p23} implies that each time a fixed point $p$ has weight $w$,  there exists a fixed point $q$ that has weight $-w$; we draw an edge from $p$ to $q$ issuing label $w$ to the edge.

\begin{definition} A \textbf{labeled directed multigraph} consists of a set $V$ of vertices, a set $E$ of edges, maps $i:E \to V$ and $t:E \to V$ giving the initial and terminal vertices of each edge, and a map $w:E \to \mathbb{N}^+$. \end{definition}

\begin{definition} Let the circle act on a compact almost complex manifold $M$ with a discrete fixed point set. We say that a labeled directed multigraph with vertex set $M^{S^1}$ \textbf{describes} $M$ if the multiset of weights at $p$ is $\{w(e)\ | \ i(e)=p\} \cup \{-w(e) \ | \ t(e)=p\}$ for all $p \in M^{S^1}$. \end{definition}

Note that this definition is slightly different from the definition in \cite{JT}. Lemma 3.6 of \cite{J2} implies the following lemma; see also Proposition 3 in \cite{JT}.

\begin{lemma} \cite{J2} \label{l24}
Let the circle act on a $2n$-dimensional compact almost complex manifold with a non-empty discrete fixed point set. Let $a$ be the smallest positive weight or the second smallest positive weight. Then for each $i$ such that $0 \leq i \leq n-1$, the number of times the weight $a$ occurs at fixed points which have exactly $i$ negative weights is equal to the number of times the weight $-a$ occurs at fixed points which have exactly $i+1$ negative weights.
\end{lemma}

Here, our convention is that the second smallest positive weight may equal the smallest positive weight and there may be other weights that are equal to the second smallest positive weight. Applying Lemma \ref{l24} to the smallest positive weight, we obtain the following lemma, which is Lemma 2.7 in \cite{J3}.

\begin{lemma} \label{l29}
Let the circle act on a compact almost complex manifold with a non-empty discrete fixed point set such that $\dim M>0$. Then there exists $i$ such that $N_i \neq 0$ and $N_{i+1} \neq 0$, where $N_j$ is the number of fixed points that have exactly $j$ negative weights. Alternatively, there exists $i$ such that $\chi^i(M) \neq 0$ and $\chi^{i+1}(M) \neq 0$. \end{lemma}

An edge $e$ of a multigraph is called a \textbf{loop} if $i(e)=t(e)$. The following lemma states that there exists a multigraph describing $M$ that does not have any loop; it satisfies additional properties.

\begin{lem} \cite{J4} \label{l28}
Let the circle act on a compact almost complex manifold $M$ with a discrete fixed point set. Then there exists a labeled directed multigraph describing $M$ with the following properties:
\begin{enumerate}[(1)]
\item Given an edge $e$, if $w(e)$ is smaller than or equal to the second smallest positive weight, then $n_{i(e)}+1=n_{t(e)}$, where $n_p$ is the number of negative weights at $p$.
\item Given an edge $e$, if $w(e)$ is strictly bigger than the second smallest positive weight, then the weights at $i(e)$ and the weights at $t(e)$ are equal modulo $w(e)$.
\item The graph has no loops. \end{enumerate} \end{lem}

In \cite{PT}, Pelayo and Tolman proved some property for a symplectic circle action with isolated fixed points in terms of the Chern class map. The proof naturally extends to a circle action on an almost complex manifold. The Chern class map at a fixed point $p$ is nothing but the sum of weights at $p$ and hence we obtain the following lemma.

\begin{lemma} \cite{PT} \label{l29}
Let the circle act on a $2n$-dimensional compact almost complex manifold with a discrete fixed point set. If the set $\{\sum_{i=1}^n w_{p,i} \ | \ p \in M^{S^1}\}$ contains at most $n$ elements, then
\begin{center}
$\displaystyle \sum_{p \in M^{S^1}, \sum_{i=1}^n w_{p,i}=k} \frac{1}{\prod_{i=1}^n w_{p,i}}=0$ for all $k \in \mathbb{Z}$.
\end{center}
\end{lemma}

To prove Theorem \ref{t11}, we need the classification results for the cases of few fixed points.

\begin{theorem} \cite{J3} \label{t211}
Let the circle act on a compact almost complex manifold $M$.
\begin{enumerate}
\item If there is exactly one fixed point, $M$ is the point.
\item If there are exactly two fixed points, either $M$ is the 2-sphere, or $\dim M=6$ and the weights at the fixed points are $\{-a-b,a,b\}$ and $\{-a,-b,a+b\}$ for some positive integers $a$ and $b$.
\item If there are exactly three fixed points, $\dim M=4$ and the weights at the fixed points are $\{a+b,a\}$, $\{-a,b\}$, $\{-b,-a-b\}$ for some positive integers $a$ and $b$.
\end{enumerate}
\end{theorem}

\begin{theorem} \cite{J4} \label{t212}
Let the circle act on a compact almost complex manifold $M$ with 4 fixed points.
\begin{enumerate}
\item If $\dim M=2$, $M$ is a disjoint union of two $S^2$'s, each of which is equipped with a rotation.
\item If $\dim M=4$, the weights at the fixed points are $\{a,b\}$, $\{-a,b\}$, $\{-b,c\}$, and $\{-b,-c\}$ for some positive integers $a,b$, and $c$ such that either $a \equiv c \mod b$ or $a \equiv -c \mod b$.
\end{enumerate}
\end{theorem}

\begin{figure}
\begin{subfigure}[b][7cm][s]{.25\textwidth}
\centering
\vfill
\begin{tikzpicture}[state/.style ={circle, draw}]
\node[state] (a) {$p_1$};
\node[state] (b) [above right=of a] {$p_2$};
\node[state] (c) [above left=of b] {$p_3$};
\path (a) [->] edge node[right] {$a$} (b);
\path (b) [->] edge node[right] {$b$} (c);
\path (a) [->]edge node [right] {$a+b$} (c);
\end{tikzpicture}
\vfill
\caption{3 fixed points}\label{fig2-1}
\vspace{\baselineskip}
\end{subfigure}\qquad
\centering
\begin{subfigure}[b][7cm][s]{.30\textwidth}
\centering
\vfill
\begin{tikzpicture}[state/.style ={circle, draw}]
\node[state] (A) {$p_1$};
\node[state] (B) [above left=of A] {$p_2$};
\node[state] (C) [above right=of A] {$p_3$};
\node[state] (D) [above right=of B] {$p_4$};
\path (A) [->] edge node[right] {$a$} (B);
\path (A) [->] edge node [right] {$b$} (C);
\path (B) [->] edge node [right] {$b$} (D);
\path (C) [->] edge node [right] {$c$} (D);
\end{tikzpicture}
\vfill
\caption{$\dim M=4$, $|M^{S^1}|=4$, I}\label{fig2-2}
\vspace{\baselineskip}
\end{subfigure}\qquad
\begin{subfigure}[b][7cm][s]{.30\textwidth}
\centering
\vfill
\begin{tikzpicture}[state/.style ={circle, draw}]
\node[state] (A) {$p_1$};
\node[state] (B) [above left=of A] {$p_2$};
\node[state] (C) [above =of B] {$p_3$};
\node[state] (D) [above right=of C] {$p_4$};
\path (A) [->] edge node[right] {$a$} (B);
\path (A) [->] edge node [right] {$b$} (D);
\path (B) [->] edge node [right] {$b$} (C);
\path (C) [->] edge node [right] {$c$} (D);
\end{tikzpicture}
\vfill
\caption{$\dim M=4$, $|M^{S^1}|=4$, II}\label{fig2-3}
\end{subfigure}
\caption{Multigraphs in Theorem \ref{t21}.}\label{fig2}
\end{figure}
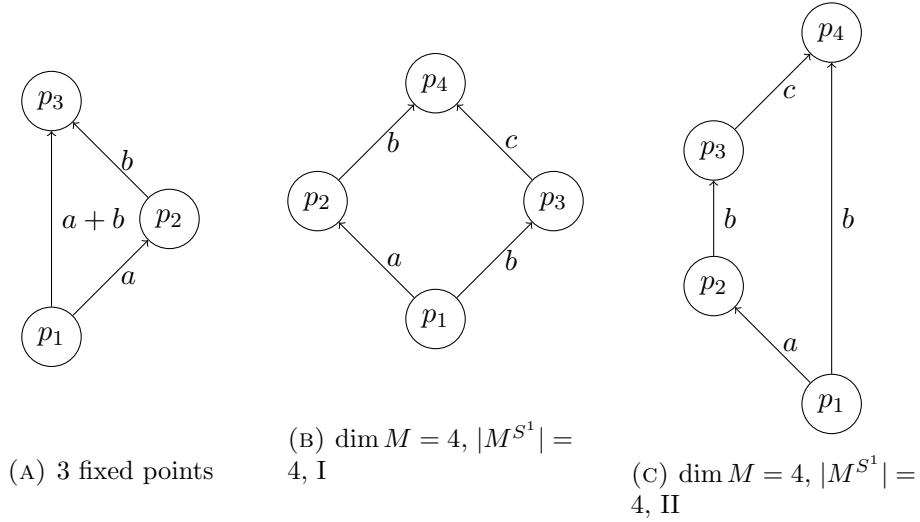

Consider a circle action on a compact almost complex manifold $M$. If there are 3 fixed points, Figure \ref{fig2-1} is the multigraph describing $M$. If $\dim M=4$ and there are 4 fixed points, Figures \ref{fig2-2} and \ref{fig2-3} are the multigraphs describing $M$, depending on whether $a \equiv c \mod b$ (Figure \ref{fig2-2}) or whether $a \equiv -c \mod b$ (Figure \ref{fig2-3}). We will need these multigraphs to prove Theorem \ref{t11}. We also need the following variant of Lemma 3.1 in \cite{TW}, which can be proved by quotienting out by the subgroup $\mathbb{Z}_w$ of $S^1$ that acts trivially.

\begin{lemma} \cite{TW}, \cite{J4} \label{l213}
Let the circle act on a $2n$-dimensional compact almost complex manifold with a non-empty discrete fixed point set. Suppose that every weight at each fixed point is either $+w$ or $-w$ for some positive integer $w$. Let $N_k$ denote the number of fixed points that have exactly $k$ negative weights. Then for all $k$ with $0 \leq k \leq n$, $N_k=N_0 \frac{n!}{k!(n-k)!}$. In particular, the total number of fixed points is equal to $N_0 \cdot 2^n$.
\end{lemma}

\section{Proof of Theorem \ref{t11}} \label{s3}

In this section, we prove our main result, Theorem \ref{t11}. The proof of Theorem \ref{t11} consists of a number of steps.

\begin{lemma} \label{l31}
Let the circle act on an 8-dimensional compact almost complex manifold $M$ with 4 fixed points. Then either $\chi_y(M)=1-y-y^3+y^4$, or $\chi_y(M)=-y+2y^2-y^3$.
\end{lemma}

\begin{proof}
For each $i$, let $N_i$ be the number of fixed points that have exactly $i$ negative weights. By Theorem \ref{t21}, $N_0=N_4$ and $N_1=N_3$. By Lemma \ref{l29}, there exists $i \in \{0,1,2,3\}$ such that $N_i \neq 0$ and $N_{i+1} \neq 0$. Since $N_0+N_1+N_2+N_3+N_4=4$ and $N_i \geq 0$ for $i=0,1,\cdots,4$, these imply that either
\begin{enumerate}
\item $N_0=1$, $N_1=1$, $N_2=0$, $N_3=1$, and $N_4=1$, or
\item $N_0=0$, $N_1=1$, $N_2=2$, $N_3=1$, and $N_4=0$.
\end{enumerate}
Since $\chi_y(M)=\sum_{i=0}^4 \chi^i(M) \cdot y^i$ and $\chi^i(M)=(-1)^i N_i$ for each $i$, the lemma follows. \end{proof}

Applying Lemma \ref{l29}, we obtain the following lemma.

\begin{lemma} \label{l32}
Let $n \geq 4$. Let the circle act on a $2n$-dimensional compact almost complex manifold $M$ with 4 fixed points. Then one of the following holds:
\begin{enumerate}[(1)]
\item $\displaystyle \sum_{i=1}^n w_{p,i}=0$ for all $p \in M^{S^1}$ and $\displaystyle \sum_{p \in M^{S^1}} \frac{1}{\prod_{i=1}^n w_{p,i}}=0$.
\item We can divide the fixed points into 2 pairs $(p_1,p_2)$ and $(p_3,p_4)$ such that
\begin{enumerate}[(a)]
\item $\displaystyle  \prod_{i=1}^n w_{p_1,i}=-\prod_{i=1}^n w_{p_2,i}$,
\item $\displaystyle  \prod_{i=1}^n w_{p_3,i}=-\prod_{i=1}^n w_{p_4,i}$, and
\item $\displaystyle  \sum_{i=1}^n w_{p_1,i}=\sum_{i=1}^n w_{p_2,i}=-\sum_{i=1}^n w_{p_3,i}=-\sum_{i=1}^n w_{p_4,i}$.
\end{enumerate}
\end{enumerate}
\end{lemma}

\begin{proof} Since $\dim M =2n \geq 8$ and the set $\{\sum_{i=1}^n w_{p,i} \ | \ p \in M^{S^1}\}$ contains at most 4 elements, by Lemma \ref{l29},
\begin{center}
$\displaystyle \sum_{p \in M^{S^1},\ \sum_{i=1}^n w_{p,i}=k} \frac{1}{\prod_{i=1}^n w_{p,i}}=0$ for all $k \in \mathbb{Z}$.
\end{center}
Pick a fixed point $p_1$. Let $a=\sum_{i=1}^n w_{p_1,i}$. Since 
\begin{center}
$\displaystyle \sum_{p \in M^{S^1},\ \sum_{i=1}^n w_{p,i}=a} \frac{1}{\prod_{i=1}^n w_{p,i}}=0$, 
\end{center}
there must exist another fixed point $p_2$ such that $\sum_{i=1}^n w_{p_2,i}=a$.

Suppose that another fixed point $p_3$ also satisfies $\sum_{i=1}^n w_{p_3,i}=a$. Let $p_4$ be the remaining fixed point. If $\sum_{i=1}^n w_{p_4,i} =b \neq a$, since $p_4$ is the only fixed point such that $\sum_{i=1}^n w_{p_4,i}=b$, we cannot have 
\begin{center}
$\displaystyle \sum_{p \in M^{S^1}, \sum_{i=1}^n w_{p,i}=b} \frac{1}{\prod_{i=1}^n w_{p,i}}=0$.
\end{center}
Therefore, we must have that $\sum_{i=1}^n w_{p_4,i}=a$. This implies that
\begin{center}
$\displaystyle \sum_{p \in M^{S^1},\ \sum_{i=1}^n w_{p,i}=a} \frac{1}{\prod_{i=1}^n w_{p,i}}=\sum_{p \in M^{S^1}} \frac{1}{\prod_{i=1}^n w_{p,i}}=0$.
\end{center}
On the other hand, by Corollary \ref{c24},
\begin{center}
$\displaystyle \sum_{p \in M^{S^1}} \sum_{i=1}^n w_{p,i}=a+a+a+a=0$.
\end{center}
Therefore, $\sum_{i=1}^n w_{p,i}=0$ for all $p \in M^{S^1}$. This is Case (1) of the lemma.

Next, suppose that another fixed point $p_3$ satisfies $\sum_{i=1}^n w_{p_3,i} \neq a$. Let $b=\sum_{i=1}^n w_{p_3,i}$. If $p_3$ is the only fixed point such that $\sum_{i=1}^n w_{p_3,i}=b$, then we have 
\begin{center}
$\displaystyle 0=\sum_{p \in M^{S^1}, \sum_{i=1}^n w_{p,i}=b} \frac{1}{\prod_{i=1}^n w_{p,i}}=\frac{1}{\prod_{i=1}^n w_{p_3,i}} \neq 0$, 
\end{center}
which is a contradiction. Therefore, the remaining fixed point $p_4$ satisfies $\sum_{i=1}^n w_{p_4,i}=b$. Since $\sum_{i=1}^n w_{p_1,i}=\sum_{i=1}^n w_{p_2,i}=a$,
\begin{center}
$\displaystyle \sum_{p \in M^{S^1}, \sum_{i=1}^n w_{p,i}=a} \frac{1}{\prod_{i=1}^n w_{p,i}}=\frac{1}{\prod_{i=1}^n w_{p_1,i}}+\frac{1}{\prod_{i=1}^n w_{p_2,i}}=0$.
\end{center}
Therefore, we have $\prod_{i=1}^n w_{p_1,i}=-\prod_{i=1}^n w_{p_2,i}$. Similarly, since $\sum_{i=1}^n w_{p_3,i}=\sum_{i=1}^n w_{p_4,i}=b$, it follows that $\prod_{i=1}^n w_{p_3,i}=-\prod_{i=1}^n w_{p_4,i}$. Finally, by Corollary \ref{c24},
\begin{center}
$\displaystyle \sum_{p \in M^{S^1}} \sum_{i=1}^n w_{p,i}=a+a+b+b=0$.
\end{center}
Hence $b=-a$. This is Case (2) of the lemma. \end{proof}

\begin{pro} \label{p33}
Let the circle act on an 8-dimensional compact almost complex manifold $M$ with 4 fixed points. Assume that the Hirzebruch $\chi_y$-genus of $M$ is $\chi_y(M)=1-y-y^3+y^4$. Then there exist positive integers $a_1,a_2,a_3,a_4,b_2,b_3,b_4,c_1$ so that exactly one of the figures in Figure \ref{fig1} occurs as a labeled directed multigraph describing $M$ that satisfies the conditions in Lemma \ref{l28}. In each case, the following inequality holds.
\begin{center}
$\max\{a_1,c_1\}<\min\{a_2,a_3,a_4,b_2,b_3,b_4\}$. 
\end{center}
Moreover, we can divide the fixed points into 2 pairs $(p_0,p_1)$ and $(p_3,p_4)$ so that
\begin{enumerate}
\item $\prod_{i=1}^4 w_{p_0,i}=-\prod_{i=1}^4 w_{p_1,i}$,
\item $\prod_{i=1}^4 w_{p_3,i}=-\prod_{i=1}^4 w_{p_4,i}$, and
\item $\sum_{i=1}^4 w_{p_0,i}=\sum_{i=1}^4 w_{p_1,i}=-\sum_{i=1}^4 w_{p_3,i}=-\sum_{i=1}^4 w_{p_4,i}$.
\end{enumerate}
\end{pro}

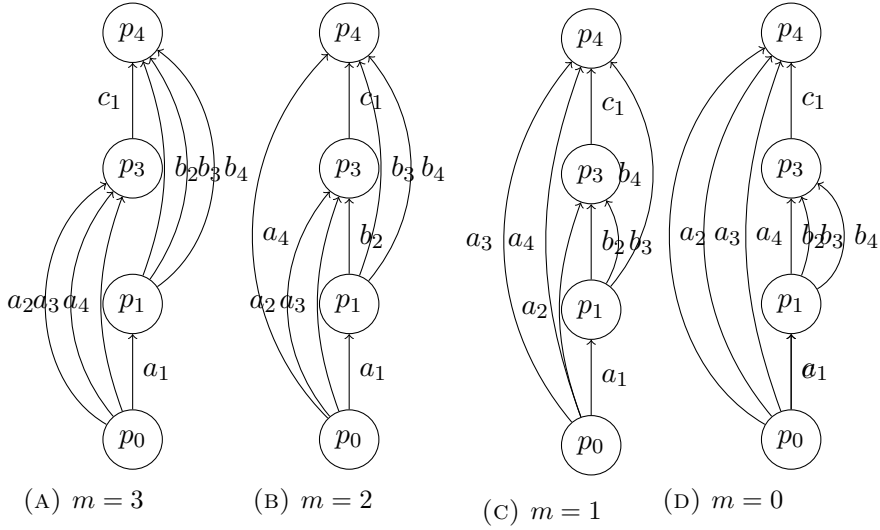
\begin{figure}
\centering
\begin{subfigure}[b][7cm][s]{.18\textwidth}
\centering
\vfill
\begin{tikzpicture}[state/.style ={circle, draw}]
\node[state] (a) {$p_0$};
\node[state] (b) [above=of a] {$p_1$};
\node[state] (c) [above=of b] {$p_3$};
\node[state] (d) [above=of c] {$p_4$};
\path (a) [->] edge node[right] {$a_1$} (b);
\path (c) [->] edge node[left] {$c_1$} (d);
\path (a) [->] [bend left =60] edge node[left] {$a_2$} (c);
\path (a) [->] [bend left =40] edge node [left] {$a_3$} (c);
\path (a) [->] [bend left =20] edge node [left] {$a_4$} (c);
\path (b) [->] [bend right =20] edge node [right] {$b_2$} (d);
\path (b) [->] [bend right =35] edge node [right] {$b_3$} (d);
\path (b) [->] [bend right =55] edge node [right] {$b_4$} (d);
\end{tikzpicture}
\vfill
\caption{$m=3$}\label{fig1-1}
\vspace{\baselineskip}
\end{subfigure}\qquad
\begin{subfigure}[b][7cm][s]{.18\textwidth}
\centering
\vfill
\begin{tikzpicture}[state/.style ={circle, draw}]
\node[state] (a) {$p_0$};
\node[state] (b) [above=of a] {$p_1$};
\node[state] (c) [above=of b] {$p_3$};
\node[state] (d) [above=of c] {$p_4$};
\path (a) [->] edge node[right] {$a_1$} (b);
\path (c) [->] edge node[right] {$c_1$} (d);
\path (a) [->] [bend left =40] edge node[left] {$a_2$} (c);
\path (a) [->] [bend left =20] edge node [left] {$a_3$} (c);
\path (a) [->] [bend left =45] edge node[right] {$a_4$} (d);
\path (b) [->] edge node [right] {$b_2$} (c);
\path (b) [->] [bend right =20] edge node [right] {$b_3$} (d);
\path (b) [->] [bend right =40]  edge node [right] {$b_4$} (d);
\end{tikzpicture}
\vfill
\caption{$m=2$}\label{fig1-2}
\vspace{\baselineskip}
\end{subfigure}\qquad
\begin{subfigure}[b][7cm][s]{.18\textwidth}
\centering
\vfill
\begin{tikzpicture}[state/.style ={circle, draw}]
\node[state] (a) {$p_0$};
\node[state] (b) [above=of a] {$p_1$};
\node[state] (c) [above=of b] {$p_3$};
\node[state] (d) [above=of c] {$p_4$};
\path (a) [->] edge node[right] {$a_1$} (b);
\path (c) [->] edge node[right] {$c_1$} (d);
\path (a) [->] [bend left =20] edge node [left] {$a_2$} (c);
\path (a) [->] [bend left =40]edge node [left] {$a_3$} (d);
\path (a) [->] [bend left =20]edge node [left] {$a_4$} (d);
\path (b) [->] edge node [right] {$b_2$} (c);
\path (b) [->] [bend right =30]  edge node [right] {$b_3$} (c);
\path (b) [->] [bend right =40]  edge node [left] {$b_4$} (d);
\end{tikzpicture}
\vfill
\caption{$m=1$}\label{fig1-3}
\end{subfigure}
\begin{subfigure}[b][7cm][s]{.18\textwidth}
\centering
\vfill
\begin{tikzpicture}[state/.style ={circle, draw}]
\node[state] (a) {$p_0$};
\node[state] (b) [above=of a] {$p_1$};
\node[state] (c) [above=of b] {$p_3$};
\node[state] (d) [above=of c] {$p_4$};
\path (a) [->] edge node[right] {$a_1$} (b);
\path (c) [->] edge node[right] {$c_1$} (d);
\path (a) [->] [bend left =60] edge node[right] {$a_2$} (d);
\path (a) [->] [bend left =40] edge node [right] {$a_3$} (d);
\path (a) [->] [bend left =20] edge node [right] {$a_4$} (d);
\path (a) [->] edge node [right] {$c$} (b);
\path (b) [->] [bend right=60] edge node [right] {$b_4$} (c);
\path (b) [->] [bend right =20]  edge node [right] {$b_3$} (c);
\path (b) [->] edge node [right] {$b_2$} (c);
\end{tikzpicture}
\vfill
\caption{$m=0$}\label{fig1-4}
\vspace{\baselineskip}
\end{subfigure}\qquad
\caption{Possible multigraphs}\label{fig1}
\end{figure}

\begin{proof} Since $\chi_y(M)=\sum_{i=0}^4 \chi^i(M) \cdot y^i=1-y-y^3+y^4$ and $\chi^i(M)=(-1)^i N_i$ for all $i$, where $N_i$ is the number of fixed points with exactly $i$ negative weights, for $i \in \{0,1,3,4\}$ there exsits exactly one fixed point $p_i$ that has exactly $i$ negative weights. Let $\{a_1,a_2,a_3,a_4\}$, $\{-b_1,b_2,b_3,b_4\}$, $\{c_1,-c_2,-c_3,-c_4\}$, $\{-d_1,-d_2,-d_3,-d_4\}$ be the weights at $p_0,p_1,p_3,p_4$, respectively, for some positive integers $a_i$, $b_i$, $c_i$, $d_i$, $1 \leq i \leq 4$. Without loss of generality, assume that $a_1 \leq a_2, a_3, a_4$ and $d_1 \leq d_2, d_3, d_4$.

We apply Lemma \ref{l28}; there exists a labeled directed multigraph $\Gamma$ discribing $M$ that satisfies the properties in Lemma \ref{l28}. Consider Property (1) in Lemma \ref{l28}. Property (1) in Lemma \ref{l28} states that $\Gamma$ has at least two edges $e$ such that $n_{i(e)}+1=n_{t(e)}$, where $n_p$ denotes the number of negative weights at $p$, and $w(e)$ is either the smallest positive weight or the second smallest positive weight. The manifold $M$ has no fixed point that has exactly two negative weights, $p_1$ has only one negative weight, and $p_3$ has only one positive weight. These imply that
\begin{enumerate}
\item There must be an edge from $p_0$ to $p_1$; $a_1$ is the label of the edge.
\item There must be an edge from $p_3$ to $p_4$; $c_1$ is the label of the edge.
\item The weights $\{a_1,c_1\}$ are the smallest positive weight and the second smallest positive weight.
\end{enumerate}
This means that $a_1=b_1$ and $c_1=d_1$. Property (1) of Lemma \ref{l28} also implies that 
\begin{center}
$\max\{a_1,c_1\} < \min\{a_2,a_3,a_4,b_2,b_3,b_4,c_2,c_3,c_4,d_2,d_3,d_4\}$, 
\end{center}
since all the other edges $e$ cannot satisfy $n_{i(e)}+1=n_{t(e)}$. This proves the second claim of the proposition.

Let $m$ be the number of edges from $p_0$ to $p_3$ in $\Gamma$. Completing the multigraph $\Gamma$, exactly one of the figures in Figure \ref{fig1} occurs as a multigraph $\Gamma$ describing $M$ that satisfies the properties in Lemma \ref{l28}. The cases $m=3,2,1,0$ correspond to Figures \ref{fig1-1}, \ref{fig1-2}, \ref{fig1-3}, \ref{fig1-4}, respectively. This proves the first claim of the proposition.

We apply Lemma \ref{l32}. Since $\sum_{i=1}^4 w_{p_0,i}>0$, Case (2) of Lemma \ref{l32} holds. Since $b_1 < b_2,b_3,b_4$, $\sum_{i=1}^4 w_{p_1,i}>0$. Similarly, since $c_1<c_2,c_3,c_4$, $\sum_{i=1}^4 w_{p_3,i}<0$. Therefore, Lemma \ref{l32} implies that
\begin{enumerate}
\item[(a)] $\prod_{i=1}^4 w_{p_0,i}=-\prod_{i=1}^4 w_{p_1,i}$,
\item[(b)] $\prod_{i=1}^4 w_{p_3,i}=-\prod_{i=1}^4 w_{p_4,i}$, and
\item[(c)] $\sum_{i=1}^4 w_{p_0,i}=\sum_{i=1}^4 w_{p_1,i}=-\sum_{i=1}^4 w_{p_3,i}=-\sum_{i=1}^4 w_{p_4,i}$.
\end{enumerate}
This proves the third claim of the proposition. \end{proof}

We are ready to determine the Hirzebruch $\chi_y$-genus of an 8-dimensional compact almost complex manifold, equipped with a circle action having 4 fixed points.

\begin{theorem} \label{t33}
Let the circle act on an 8-dimensional compact almost complex manifold $M$ with 4 fixed points. Then the Hirzebruch $\chi_y$-genus of $M$ is $\chi_y(M)=-y+2y^2-y^3$.
\end{theorem}

\begin{proof} By Lemma \ref{l31}, either $\chi_y(M)=1-y-y^3+y^4$, or $\chi_y(M)=-y+2y^2-y^3$. To prove the theorem, assume on the contrary that $\chi_y(M)=1-y-y^3+y^4$. By Proposition \ref{p33}, exactly one of the figures in Figure \ref{fig1} occurs as a labeled directed multigraph $\Gamma$ describing $M$ that satisfies the conditions in Lemma \ref{l28}.

Suppose that Figure \ref{fig1-1} is the multigraph $\Gamma$ describing $M$ that satisfies the conditions in Lemma \ref{l28}. By Property (3) of Proposition \ref{p33}, we have $a_1+a_2+a_3+a_4=\sum_{i=1}^4 w_{p_0,i} = -\sum_{i=1}^4 w_{p_3,i}=-c_1+a_2+a_3+a_4$, but this cannot hold. Therefore, Figure \ref{fig1-1} cannot be a multigraph describing $M$.

Suppose that Figure \ref{fig1-2} is the multigraph $\Gamma$ describing $M$ that satisfies the conditions in Lemma \ref{l28}. By Property (3) of Proposition \ref{p33}, we have that $a_1+a_2+a_3+a_4=-a_1+b_2+b_3+b_4=-c_1+a_2+a_3+b_2=a_4+b_3+b_4+c_1$. From this, it follows that $b_2=a_1+a_4+c_1$. Moreover, by Property (1) of Proposition \ref{p33}, we have that $a_1 a_2 a_3 a_4=\prod_{i=1}^4 w_{p_0,i}=-\prod_{i=1}^4 w_{p_1,i}=a_1 b_2 b_3 b_4$ and hence $a_2 a_3 a_4=b_2 b_3 b_4$. By Property (2) of Proposition \ref{p33}, we have that $-a_2 a_3 b_2 c_1=\prod_{i=1}^4 w_{p_3,i}=-\prod_{i=1}^4 w_{p_4,i}=-a_4 b_3 b_4 c_1$ and hence $a_2 a_3 b_2=a_4 b_3 b_4$. The two equations $a_2 a_3 a_4=b_2 b_3 b_4$ and $a_2 a_3 b_2=a_4 b_3 b_4$ imply that $a_4=b_2$ but $b_2=a_1+a_4+c_1>a_4$. Therefore, Figure \ref{fig1-2} cannot be a multigraph describing $M$.

Suppose that Figure \ref{fig1-3} is the multigraph $\Gamma$ describing $M$ that satisfies the conditions in Lemma \ref{l28}. By Property (3) of Proposition \ref{p33}, we have $a_1+a_2+a_3+a_4=-a_1+b_2+b_3+b_4=-c_1+a_2+b_2+b_3=a_3+a_4+b_4+c_1$. This implies that $a_1+a_2=b_4+c_1$ and $2a_1+a_2+a_3+a_4=b_2+b_3+b_4$. Next, Property (1) of Proposition \ref{p33} implies that $a_1 a_2 a_3 a_4=\prod_{i=1}^4 w_{p_0,i}=-\prod_{i=1}^4 w_{p_1,i}=a_1 b_2 b_3 b_4$ and hence $a_2 a_3 a_4=b_2 b_3 b_4$. By Property (2), we have that $-a_2 b_2 b_3 c_1=\prod_{i=1}^4 w_{p_3,i}=-\prod_{i=1}^4 w_{p_4,i}=-a_3 a_4 b_4 c_1$ and hence $a_2 b_2 b_3=a_3 a_4 b_4$. The equations $a_2 a_3 a_4=b_2 b_3 b_4$ and $a_2 b_2 b_3=a_3 a_4 b_4$ imply that $a_2=b_4$ and $a_3 a_4=b_2 b_3$. Since $a_1+a_2=b_4+c_1$, this implies that $a_1=c_1$. From $2a_1+a_2+a_3+a_4=b_2+b_3+b_4$, this also implies that $2a_1+a_3+a_4=b_2+b_3$. The weights at the fixed points are than $\Sigma_{p_0}=\{a_1,a_2,a_3,a_4\}$, $\Sigma_{p_1}=\{-a_1,a_2,b_2,b_3\}$, $\Sigma_{p_3}=\{-a_2,-b_2,-b_3,a_1\}$, $\Sigma_{p_4}=\{-a_1,-a_2,-a_3,-a_4\}$.

We push-forward the equivariant second Chern class $c_2(M)$ in equivariant cohomology. By Theorem \ref{t21},
\begin{center}
$\displaystyle 0=\int_M c_2(M)=\sum_{p \in M^{S^1}} \int_p \frac{c_2(M)|_p}{e_{S^1}(N_F)}=\sum_{p \in M^{S^1}} \frac{\sigma_2(w_{p,1},w_{p,2},w_{p,3},w_{p,4})}{\prod_{i=1}^4 w_{p,i}}$
$= \frac{a_1a_2+a_1a_3+a_1a_4+a_2a_3+a_2a_4+a_3a_4}{a_1 a_2 a_3 a_4}+\frac{-a_1a_2-a_1b_2-a_1b_3+a_2b_2+a_2b_3+b_2b_3}{-a_1 a_2 b_2 b_3} + \frac{-a_1a_2-a_1b_2-a_1b_3+a_2b_2+a_2b_3+b_2b_3}{-a_1 a_2 b_2 b_3}+\frac{a_1a_2+a_1a_3+a_1a_4+a_2a_3+a_2a_4+a_3a_4}{a_1 a_2 a_3 a_4}$.
\end{center}
Here, $\sigma_2$ is the second elementary symmetric polynomial in 4 variables. We have seen that the denominators are the same up to sign. With $2a_1+a_3+a_4=b_2+b_3$ and $a_3 a_4=b_2 b_3$, the equation above is equivalent to $a_1 a_4+a_2 a_3+a_1 b_2+a_1 b_3=0$, which cannot hold.

Suppose that Figure \ref{fig1-4} is the multigraph $\Gamma$ describing $M$ that satisfies the conditions in Lemma \ref{l28}. Without loss of generality, assume that $a_2 \leq a_3 \leq a_4$. By Property (1) of Proposition \ref{p33}, $a_1 a_2 a_3 a_4=\prod_{i=1}^4 w_{p_0,i}=-\prod_{i=1}^4 w_{p_1,i}=c_1 a_2 a_3 a_4$ and hence $a_1=c_1$.

We prove that $a_2,a_3<a_4$. Assume that $a_2=a_3=a_4$. Since $a_1<a_2$, $1<a_2$. Consider the set $M^{\mathbb{Z}_{a_2}}$ of points in $M$ that are fixed by the $\mathbb{Z}_{a_2}$-action, where $\mathbb{Z}_{a_2}$ acts on $M$ as a subgroup of $S^1$. The set $M^{\mathbb{Z}_{a_2}}$ is a union of smaller dimensional compact almost complex submanifolds. Let $Z$ be a connected component of $M^{\mathbb{Z}_{a_2}}$ that contains $p_0$. Since $p_0$ has 3 weights $a_2,a_3,a_4$ that are divisible by $a_2$, $\dim Z=6$. The circle action on $M$ restricts to a circle action $C$ on $Z$. Since $a_2=a_4$ is the largest weight, every weight in $T_pZ$ at a fixed point $p$ of the $C$-action on $Z$ is either $+a_2$ or $-a_2$. By Lemma \ref{l213}, the $C$-action on $Z$ has $\ell \cdot 2^3$ fixed points for some positive integer $\ell$. However, $Z^C \subset M^{S^1}$ but $M^{S^1}$ consists of 4 fixed points. Therefore, we cannot have $a_2=a_3=a_4$.

Next, assume that $a_2<a_3=a_4$. By an analogous argument, a connected component $Z$ of $M^{\mathbb{Z}_{a_3}}$ containing $p_0$ satisfies $\dim Z=4$ since $p_0$ has 2 weights $a_3,a_4$ that are divisible by $a_3$, and the restriction $C$ of the circle action on $Z$ has $\ell \cdot 2^2$ fixed points for some positive integer $\ell$. Since $Z^C \subset M^{S^1}$ and $M^{S^1}$ consists of 4 fixed points, this implies that $Z^C=M^{S^1}$. Next, by Lemma \ref{l213}, the weights in $T_pZ$ at the fixed points $p$ of the $C$-action on $Z$ are $\{a_3,a_3\}, \{-a_3,a_3\}, \{-a_3,a_3\}, \{-a_3,-a_3\}$. This implies that there must be a fixed point $p \in M^{S^1}$ such that $\{-a_3,a_3\} \subset \{w_{p,1},\cdots,w_{p,4}\}$. However, since $a_1=c_1<a_3$, there cannot be such a fixed point; $p_0$ does not have a negative weight, $p_1$ does not have a negative weight divisible by $a_2$ (the negative weight $-a_1$ is bigger than $-a_3$), $p_3$ does not have a positive weight divisible by $a_3$, and $p_4$ does not have a positive weight. Therefore, $a_2,a_3<a_4$.

By Property (2) of Lemma \ref{l28} for the edge whose label is $a_4$, the weights at $p_0$ and the weights at $p_4$ are equal modulo $a_4$, i.e., 
\begin{center}
$\{a_1,a_2,a_3,a_4\} \equiv \{-a_1,-a_2,-a_3,-a_4\} \mod a_4$. 
\end{center}
First, $a_4 \equiv -a_4 \mod a_4$. Next, consider $a_1$.

Suppose that $a_1 \equiv -a_1 \mod a_4$. Since $a_1<a_4$, this means that $2a_1=a_4$. Then we are left with $\{a_2,a_3\} \equiv \{-a_2,-a_3\} \mod a_4$. Since $a_2,a_3<a_4$ by the claim above, this implies that $a_2+a_3=a_4$. This leads to a contradiction since $a_4=2a_1<a_2+a_3=a_4$.

Next, suppose that $a_1 \equiv -a_2 \mod a_4$. This means that $a_1+a_2=a_4$. Then we must have $a_3 \equiv -a_3 \mod a_4$ and hence $2a_3=a_4$. This leads to a contradiction since $a_4=a_1+a_2 < 2a_3=a_4$.

Finally, suppose that $a_1 \equiv -a_3 \mod a_4$. This means that $a_1+a_3=a_4$. Then we must have $a_2 \equiv -a_2 \mod a_4$, i.e., $2a_2=a_4$. It follows that $a_1<a_2<a_3<a_4=2a_2$, and this means that a connected component $Z$ of $M^{\mathbb{Z}_{a_2}}$ containing $p_0$ satisfies $\dim Z=4$, since $p_0$ has 2 weights $a_2,a_4$ that are divisible by $a_2$. The circle action on $M$ restricts to a circle action $C$ on $Z$. Since $p_0 \in Z^C \subset M^{S^1}$, the $C$-action on $Z$ has a non-empty fixed point set and has at most 4 fixed points. Since $\dim Z=4$, by Theorem \ref{t211}, the $C$-action on $Z$ has either 3 fixed points or 4 fixed points. In either case, by Theorems \ref{t211} and \ref{t212}, there must be a fixed point in $Z$ for the $C$-action that has one positive weight and one negative weight that are both divisible by $a_2$, since every weight at a fixed point $p$ in $T_pZ$ is a multiple of $a_2$. However, since $a_1=c_1<a_2$, every fixed point $p$ of the circle action on $M$ cannot have one positive weight and one negative weight that are both divisible by $a_2$. Alternatively, one of the figures in Figure \ref{fig2} must lie in Figure \ref{fig1-4} as a subgraph in which the label of each edge is a multiple of $a_2$. Since $a_1=c_1<a_2$, this is impossible. Thus the theorem follows. \end{proof}

We have determined the Hirzebruch $\chi_y$-genus of an 8-dimensional compact almost complex manifold, equipped with a circle action having 4 fixed points. With this, we determine all the Chern numbers of the manifold.

\begin{theorem} \label{t35} Let the circle act on an 8-dimensional compact almost complex manifold $M$ with 4 fixed points. Then the Chern numbers of $M$ are 
\begin{center}
$\displaystyle \int_M c_1^4=\int_M c_1^2 c_2=\int_M c_2^2=0$, and $\displaystyle \int_M c_1 c_3=\int_M c_4=4$. 
\end{center}
In particular, $M$ is unitary cobordant to $S^2 \times S^6$. \end{theorem}

\begin{proof} The Hirzebruch $\chi_y$-genus of $M$ is $\chi_y(M)=\sum_{i=0}^4 \chi^i(M) \cdot y^i$, where $\chi^i(M)=\int_M T_i^4$ and
\begin{center}
$\displaystyle T_0^4=\frac{-c_1^4+4c_1^2c_2+3c_2^2+c_1c_3-c_4}{720}=T_4^4$,

$\displaystyle T_1^4=\frac{-c_1^4+4c_1^2c_2+3c_2^2-14c_1c_3-31c_4}{180}=T_3^4$, and

$\displaystyle T_2^4=\frac{-c_1^4+4c_1^2c_2+3c_2^2-19c_1c_3+79c_4}{120}$.
\end{center}
For these facts, see \cite{GS} for instance. By Theorem \ref{t33}, the Hirzebruch $\chi_y$-genus of $M$ is $\chi_y(M)=-y+2y^2-y^3$. Therefore, we have that $\int_M T_0^4=0$, $\int_M T_1^4=-1$, and $\int_M T_2^4=2$. On the other hand,
\begin{center}
$\displaystyle \int_M c_4=\sum_{p \in M^{S^1}} \frac{c_4(M)|_p}{\prod_{i=1}^4 w_{p,i}}=\sum_{p \in M^{S^1}} \frac{\prod_{i=1}^4 w_{p,i}}{\prod_{i=1}^4 w_{p,i}}=4$.
\end{center}
Then we have
\begin{center}
$\displaystyle 0=\int_M T_0^4=\frac{\{\int_M (-c_1^4+4c_1^2c_2+3c_2^2+c_1c_3)\}-4}{720}$,

$\displaystyle -1=\int_M T_1^4=\frac{\{\int_M (-c_1^4+4c_1^2c_2+3c_2^2-14c_1c_3)\}-124}{180}$, and

$\displaystyle 2=\int_M T_2^4=\frac{(\{\int_M (-c_1^4+4c_1^2c_2+3c_2^2-19c_1c_3)\} + 316}{120}$.
\end{center}
These imply that $\int_M c_1c_3=4$ and $\int_M (-c_1^4+ 4c_1^2 c_2+3c_2^2)=0$.

By Lemma \ref{l32}, one of the following holds:
\begin{enumerate}[(1)]
\item $\sum_{i=1}^4 w_{p,i}=0$ for all $p \in M^{S^1}$ and $\displaystyle \sum_{p \in M^{S^1}} \frac{1}{\prod_{i=1}^n w_{p,i}}=0$.
\item We can divide the fixed points into 2 pairs $(p_1,p_2)$ and $(p_3,p_4)$ such that
\begin{enumerate}[(a)]
\item $\prod_{i=1}^4 w_{p_1,i}=-\prod_{i=1}^4 w_{p_2,i}$,
\item $\prod_{i=1}^4 w_{p_3,i}=-\prod_{i=1}^4 w_{p_4,i}$, and
\item $\sum_{i=1}^4 w_{p_1,i}=\sum_{i=1}^4 w_{p_2,i}=-\sum_{i=1}^4 w_{p_3,i}=-\sum_{i=1}^4 w_{p_4,i}$.
\end{enumerate}
\end{enumerate}

Suppose that Case (1) holds. Since $c_1(M)|_p=\sum_{i=1}^4 w_{p,i}=0$ for any fixed point $p \in M^{S^1}$,
\begin{center}
$\displaystyle \int_M c_1^4=\sum_{p \in M^{S^1}} \frac{c_1^4(M)|_p}{\prod_{i=1}^4 w_{p,i}}=\sum_{p \in M^{S^1}} \frac{(\sum_{i=1}^4 w_{p,i})^4}{\prod_{i=1}^4 w_{p,i}}=0$
\end{center}
and similarly
\begin{center}
$\displaystyle \int_M c_1^2 c^2=\sum_{p \in M^{S^1}} \frac{c_1^2(M)|_p \cdot c_2(M)|_p}{\prod_{i=1}^4 w_{p,i}}=\sum_{p \in M^{S^1}} \frac{(\sum_{i=1}^4 w_{p,i})^2 \cdot c_2(M)|_p}{\prod_{i=1}^4 w_{p,i}}=0$. 
\end{center}
Since $\int_M (-c_1^4+ 4c_1^2 c_2+3c_2^2)=0$, it follows that $\int_M c_2^2=0$.

Suppose that Case (2) holds. Let 
\begin{center}
$\displaystyle a=\sum_{i=1}^4 w_{p_1,i}=\sum_{i=1}^4 w_{p_2,i}=-\sum_{i=1}^4 w_{p_3,i}=-\sum_{i=1}^4 w_{p_4,i}$, 

$\displaystyle A=\prod_{i=1}^4 w_{p_1,i}=-\prod_{i=1}^4 w_{p_2,i}$, and 

$\displaystyle B=\prod_{i=1}^4 w_{p_3,i}=-\prod_{i=1}^4 w_{p_4,i}$. 
\end{center}
Then we have
\begin{center}
$\displaystyle \int_M c_1^4=\sum_{j=1}^4 \frac{(\sum_{i=1}^4 w_{p_j,i})^4}{\prod_{i=1}^4 w_{p_j,i}}=\frac{a^4}{A}+\frac{a^4}{-A}+\frac{(-a)^4}{B}+\frac{(-a)^4}{-B}=0$.
\end{center}
Next, by a dimension reason, we have
\begin{center}
$\displaystyle 0=\int_M c_2=\frac{c_2(M)|_{p_1}}{A}+\frac{c_2(M)|_{p_2}}{-A}+\frac{c_2(M)|_{p_3}}{B}+\frac{c_2(M)|_{p_4}}{-B}$.
\end{center}
Therefore, we have
\begin{center}
$\displaystyle \int_M c_1^2 c_2=\sum_{p \in M^{S^1}} \frac{c_1^2(M)|_p \cdot c_2(M)|_p}{\prod_{i=1}^4 w_{p,i}}=\frac{a^2 \cdot c_2(M)|_{p_1}}{A}+\frac{a^2 \cdot c_2(M)|_{p_2}}{-A}+\frac{(-a)^2 \cdot c_2(M)|_{p_3}}{B}+\frac{(-a)^2 \cdot c_2(M)|_{p_4}}{-B}=a^2 \int_M c_2=0$.
\end{center}
Since $\int_M (-c_1^4+ 4c_1^2 c_2+3c_2^2)=0$, it follows that $\int_M c_2^2=0$. 

We have proved that any circle action on an 8-dimensional compact almost complex manifold with 4 fixed points has the unique Chern numbers. Therefore, any two such manifolds are equivariantly cobordant; see \cite{GGK}. In particular, since $S^2 \times S^6$ is such a manifold, $M$ is equivariantly cobordant to $S^2 \times S^6$. \end{proof}

\section{$S^2 \times S^6$ and $S^6 \times S^6$, and discussion for higher dimensions} \label{s4}

\begin{figure}
\begin{subfigure}[b][4.5cm][s]{.10\textwidth}
\centering
\vfill
\begin{tikzpicture}[state/.style ={circle, draw}]
\node[state] (a) {$p_1$};
\node[state] (b) [above=of a] {$p_2$};
\path (a) [->]edge node [right] {$a$} (b);
\end{tikzpicture}
\vfill
\caption{$S^2$}\label{fig3-1}
\vspace{\baselineskip}
\end{subfigure}\qquad
\begin{subfigure}[b][4.5cm][s]{.10\textwidth}
\centering
\vfill
\begin{tikzpicture}[state/.style ={circle, draw}]
\node[state] (a) {$p_1$};
\node[state] (b) [above=of a] {$p_2$};
\path (a) [->] [bend left =5]edge node[left] {$c$} (b);
\path (a) [->] [bend left =30]edge node[left] {$b$} (b);
\path (b) [->] [bend left =10]edge node[right] {$b+c$} (a);
\end{tikzpicture}
\vfill
\caption{$S^6$}\label{fig3-2}
\vspace{\baselineskip}
\end{subfigure}\qquad
\begin{subfigure}[b][4.5cm][s]{.30\textwidth}
\centering
\vfill
\begin{tikzpicture}[state/.style ={circle, draw}]
\node[state] (a) {$p_1$};
\node[state] (b) [above=of a] {$p_2$};
\node[state] (c) [right=of a]{$p_3$};
\node[state] (d) [above=of c] {$p_4$};
\path (a) [->] [bend left =5]edge node[left] {$c$} (b);
\path (a) [->] [bend left =30]edge node[left] {$b$} (b);
\path (b) [->] [bend left =15]edge node[right] {$b+c$} (a);
\path (c) [->] [bend left =5]edge node[left] {$c$} (d);
\path (c) [->] [bend left =30]edge node[left] {$b$} (d);
\path (d) [->] [bend left =15]edge node[right] {$b+c$} (c);
\path (a) [->] edge node[below] {$a$} (c);
\path (b) [->] edge node[above] {$a$} (d);
\end{tikzpicture}
\vfill
\caption{$S^2 \times S^6$}\label{fig3-3}
\vspace{\baselineskip}
\end{subfigure}\qquad
\begin{subfigure}[b][4.5cm][s]{.30\textwidth}
\centering
\vfill
\begin{tikzpicture}[state/.style ={circle, draw}]
\node[state] (a) {$p_1$};
\node[state] (b) [above=of a] {$p_2$};
\node[state] (c) [right=of a]{$p_3$};
\node[state] (d) [above=of c] {$p_4$};
\path (a) [->] [bend left =5]edge node[left] {$b_1$} (b);
\path (a) [->] [bend left =45]edge node[left] {$a_1$} (b);
\path (b) [->] [bend left =15]edge node[right] {$a_1+b_1$} (a);
\path (c) [->] [bend right =5]edge node[right] {$b_1$} (d);
\path (c) [->] [bend right =45]edge node[right] {$a_1$} (d);
\path (d) [->] [bend right =15]edge node[left] {$a_1+b_1$} (c);
\path (a) [->] [bend left =5]edge node[below] {$b_2$} (c);
\path (a) [->] [bend left =35]edge node[below] {$a_2$} (c);
\path (c) [->] [bend left =45]edge node[below] {$a_2+b_2$} (a);
\path (b) [->] [bend left =10]edge node[above] {$b_2$} (d);
\path (b) [->] [bend left =45]edge node[above] {$a_2$} (d);
\path (d) [->] [bend left =10]edge node[below] {$a_2+b_2$} (b);
\end{tikzpicture}
\vfill
\caption{$S^6 \times S^6$}\label{fig3-4}
\vspace{\baselineskip}
\end{subfigure}\qquad
\caption{Multigraphs for $S^2$, $S^6$, $S^2 \times S^6$, and $S^6 \times S^6$}\label{fig3}
\end{figure}
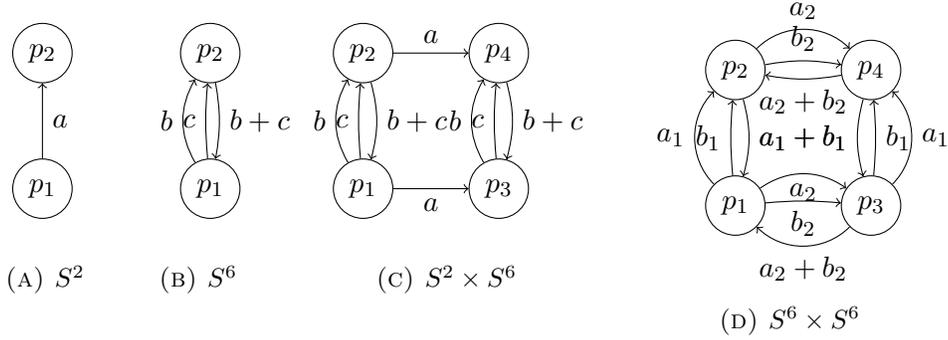

Rotating the 2-sphere $S^2$ $a$-times has two fixed points, the north pole and the south pole. The weight at the north pole is $-a$ and the weight at the south pole is $a$. The multigraph describing the action is Figure \ref{fig3-1}.

For $S^6$, there exists a circle action on $S^6$ with 2 fixed points. The weights at the fixed points are $\{-b-c,b,c\}$ and $\{-b,-c,b+c\}$ for some positive integers $b$ and $c$. The multigraph describing the action on $S^6$ is Figure \ref{fig3-2}. Then the product manifold $S^2 \times S^6$ equipped with a diagonal action has 4 fixed points, and the weights at the fixed points are $\{-b-c,a,b,c\}$, $\{-b,-c,a,b+c\}$, $\{-a,-b-c,b,c\}$, and $\{-a,-b,-c,b+c\}$. The multigraph describing $S^2 \times S^6$ is Figure \ref{fig3-3}. From the weights at the fixed points, one can compute the Hirzebruch $\chi_y$-genus of $S^2 \times S^6$ and the Chern numbers $\int c_1^4$, $\int c_1^2 c_2$, $\int c_2^2$, $\int c_1c_3$, $\int c_4$, and check that these agree with those in Theorem \ref{t11}.

For $i=1,2$, let $M_i$ be $S^6$ equipped with a circle action having 2 fixed points whose weights are $\{-a_i-b_i,a_i,b_i\}$ and $\{-a_i,-b_i,a_i+b_i\}$ for some positive integers $a_i$ and $b_i$. Then the product manifold $S^6 \times S^6 = M_1 \times M_2$ equipped with a diagonal action has 4 fixed points that have weights $\{-a_1-b_1,a_1,b_1,-a_2-b_2,a_2,b_2\}$, $\{-a_1-b_1,a_1,b_1,-a_2,-b_2,a_2+b_2\}$, $\{-a_1,-b_1,a_1+b_1,-a_2-b_2,a_2,b_2\}$, $\{-a_1,-b_1,a_1+b_1,-a_2,-b_2,a_2+b_4\}$. Figure \ref{fig3-4} is the multigraph that describes this action. The manifold $S^6 \times S^6$ is the only known example of a 12-dimensional manifold with 4 fixed points, to the author's knowledge.

For the action on $S^2 \times S^6$ above, the first two fixed points have $\sum_{i=1}^4 w_{p,i}=a$ and the next two fixed points have $\sum_{i=1}^4 w_{p,i}=-a$. This is Case (2) of Lemma \ref{l32}. Property (2-c) of Lemma \ref{l32} says that the product of the weights at $p_1$ and the product of the weights at $p_2$ are the negative of each other. Moreover, the weights at $p_1$ and the weights at $p_2$ agree up to sign. Similarly, the weights at $p_3$ and the weights $p_4$ agree up to sign. In fact, the weights at any two fixed points agree up to sign. And this phenomenon is also true for a diagonal action on $S^6 \times S^6$ with 4 fixed points. Therefore, for a higher dimensional manifold with 4 fixed points, one may ask if this is a common phenomena: Let $n \geq 4$. Let the circle act on a $2n$-dimensional compact almost complex manifold $M$ with 4 fixed points. Then for any pair $(p,q)$ of fixed points, do the weights at $p$ and the weights at $q$ agree up to sign? If the answer is affirmative, this would simplify the classification for higher dimensional manifolds with 4 fixed points.




\end{document}